\documentclass[10pt,psamsfonts,twoside]{article}
\usepackage{amsfonts}
\usepackage{amsfonts,latexsym,bm}
\usepackage{amsmath,amsthm}
\usepackage{amssymb,amscd}
\usepackage{amsfonts,amsbsy}
\usepackage{fancyhdr,graphicx}
\usepackage[dvips]{psfrag}
\usepackage{indentfirst}
\usepackage{xcolor}
\numberwithin{equation}{section}
\newtheorem {theorem} {Theorem}[section]
\newtheorem {proposition} [theorem]{Proposition}
\newtheorem {corollary} [theorem]{Corollary}
\newtheorem {definition} [theorem]{Definition}
\newtheorem {lemma}  [theorem]{Lemma}

\newcommand{\diva}{{\rm div}}
\newcommand{\supp}{{\rm supp}}

\begin{document}
\setlength{\parindent}{4ex} \setlength{\parskip}{1ex}
\setlength{\oddsidemargin}{12mm} \setlength{\evensidemargin}{9mm}

\title{{Global Existence and Stability for a Hydrodynamic System in the
Nematic Liquid Crystal Flows}\footnote{This work is supported by the
China National Natural Science Fundation under the Grant No.
10771223.}}
\author{Jihong Zhao\footnote{E-mail:
zhaojihong2007@yahoo.com.cn.},  \ Qiao Liu\footnote{E-mail:
liuqao2005@163.com.},\  and\  Shangbin Cui\footnote{E-mail:
cuisb3@yahoo.com.cn.}\\ [0.2cm] {\small Department of Mathematics,
Sun Yat-sen
University, Guangzhou, }\\
{\small Guangdong 510275, People's Republic of China}}
\date{}
\maketitle

\begin{abstract}
In this paper we consider a coupled hydrodynamical system which
involves the Navier-Stokes equations for the velocity field and
kinematic transport equations for the molecular orientation field.
By applying the Chemin-Lerner's time-space estimates for the heat
equation and the Fourier localization technique, we prove that when
initial data belongs to the critical Besov spaces with
negative-order, there exists a unique local solution, and this
solution is global when initial data is small enough. As a
corollary, we obtain existence of global self-similar solutions. In
order to figure out the relation between the solution obtained here
and weak solution of  standard sense, we establish a stability
result, which yields in a direct way that all global weak solutions
associated with the same initial data must coincide with the
solution obtained here, namely, weak-strong uniqueness holds.


\end{abstract}

\section{\bf Introduction}

In this paper, we study the following hydrodynamical system modeling
the flow of nematic liquid crystal in $\mathbb{R}^{n}$:
\begin{align}\label{eq1.1}
  &\partial_{t} \mathbf{u}-\nu\Delta \mathbf{u} +\mathbf{u}\cdot\nabla\mathbf{u}+\nabla{P}
  =-\lambda\ \text{div }(\nabla \mathbf{d} \odot\nabla \mathbf{d}),\\
\label{eq1.2}
  &\partial_{t}  \mathbf{d}+\mathbf{u}\cdot\nabla\mathbf{d}=\gamma(\Delta \mathbf{d}-f(\mathbf{d})),\\
\label{eq1.3}
  &{\rm \text{div }}\mathbf{u}=0,\\
\label{eq1.4}
  &(\mathbf{u}, \mathbf{d})|_{t=0}=(\mathbf{u}_{0}, \mathbf{d}_{0}),
\end{align}
 where $\mathbf{u}$ and $P$ denote the velocity field and the
pressure of the flow, respectively, $\mathbf{d}$ denotes the
(averaged) macroscopic/continuum molecule orientation field,
$\nu,\lambda,\gamma$ are positive constants, and $f(\mathbf{d})$ is
a Ginzburg-Landau approximation function.  The notation $\nabla
\mathbf{d}\odot\nabla \mathbf{d}$ denotes the $n\times n$ matrix
whose $(i,j)$-th entry is given by $\partial_{i}\mathbf{d}\cdot
\partial_{j}\mathbf{d}$ ($1\leq i,j\leq n$). As in
\cite{HW10} and \cite{LL95}, we assume $f(\mathbf{d})=0$ for
simplicity. Besides, since the size of the viscosity constants
$\nu$, $\lambda$ and $\gamma$ do not play a special role in our
discussion, we assume that they all equal to the unit.

The above system \eqref{eq1.1}--\eqref{eq1.4} describes the time
evolution of nematic liquid crystal materials (cf. \cite{L89}).
Equation \eqref{eq1.1} is the conservation of linear momentum (the
force balance equation). Equation \eqref{eq1.2} is the conservation
of angular momentum, in which the left hand side represents the
kinematic transport by the flow field, while the right hand side
represents the internal relaxation due to the elastic energy.
Finally, equation \eqref{eq1.3} represents the incompressibility of
the fluid. This system was first introduced by Lin \cite{L89} as a
simplified version of the liquid crystal model proposed by Ericksen
in \cite{E61} and Leslie in \cite{L68}, and retained most of the
interesting mathematical properties of the liquid crystal model.
Some basic results concerning the mathematical theory of this system
were obtained by Lin and Liu in \cite{LL95} and \cite{LL96}. More
precisely, in \cite{LL95} they proved global existence of weak
solutions by using the modified Galerkin method combined with some
compactness argument. Moreover, they also proved global existence of
strong solutions if the initial data is sufficiently small (or if
the viscosity $\nu$ is sufficiently large). In \cite{LL96} they
proved that the one-dimensional space-time Hausdorff measure of the
singular set of ``suitable" weak solutions is zero. Recently, by
using the maximal regularity of Stokes equations and the parabolic
equations, Hu and Wang \cite{HW10} proved global existence of strong
solutions to the system \eqref{eq1.1}--\eqref{eq1.4} for initial
data belonging to Besov spaces of positive-order under the smallness
assumption. Here we also refer the reader to see \cite{E87},
\cite{HK87}, \cite{L79}, \cite{LLW10}, \cite{LW08}, \cite{LWX09},
\cite{H11}, \cite{SL09} and the references therein for more details
of the physical background of this problem and some different models
of similar equations.

As in the work of Hu and Wang \cite{HW10}, we let $F=\nabla
\mathbf{d}$. Then, taking the gradient of \eqref{eq1.2}, noticing
the facts that $F\odot F=F^{T}F$ ($F^{T}$ denotes the transpose of
$F$) and
$$
  \frac{\partial}{\partial x_{k}}\Big(\sum_{j=1}^{n}\mathbf{u}_{j}\frac{\partial \mathbf{d}_{i}}{\partial
  x_{j}}\Big)=\sum_{j=1}^{n}\frac{\partial \mathbf{u}_{j}}{\partial x_{k}}\frac{\partial \mathbf{d}_{i}}{\partial
  x_{j}}+\sum_{j=1}^{n}\mathbf{u}_{j}\frac{\partial}{\partial x_{j}}\Big(\frac{\partial \mathbf{d}_{i}}{\partial
  x_{k}}\Big)=(F\nabla \mathbf{u}+\mathbf{u}\cdot\nabla F)_{ik}
$$
for all $i, k=1,2, \cdots, n$, and applying the Leray-Hopf projector
$\mathbb{P}$ to eliminate the pressure $P$, we see that the system
\eqref{eq1.1}--\eqref{eq1.4} can be reduced into the following
system:
\begin{align}\label{eq1.5}
  &\partial_{t} \mathbf{u}-\Delta \mathbf{u}
  =-\mathbb{P}\mathbf{u}\cdot\nabla\mathbf{u}-\mathbb{P}\diva(F^{T}F),\\
\label{eq1.6}
  &\partial_{t} F-\Delta F=-\mathbf{u}\cdot\nabla F-F\nabla
  \mathbf{u},\\
\label{eq1.7}
  &(\mathbf{u},F)|_{t=0}=(\mathbf{u}_{0},F_{0}),
\end{align}
where $F_{0}=\nabla \mathbf{d}_{0}$. Recall that
$\mathbb{P}=I+\nabla(-\Delta)^{-1}\diva$, i.e., $\mathbb{P}$ is the
$n\times n$ matrix pseudo-differential operator in $\mathbb{R}^{n}$
with the symbol
$(\delta_{ij}-\frac{\xi_i\xi_j}{|\xi|^2})_{i,j=1}^{n}$, where $I$
represents the unit operator and $\delta_{ij}$ is the Kronecker
symbol. Later on we shall consider the Cauchy problem
\eqref{eq1.5}--\eqref{eq1.7}.

The purpose of this paper is to prove global existence and stability
of solutions to the problem \eqref{eq1.5}--\eqref{eq1.7} in the
critical Besov space $\dot{B}^{-1+n/p}_{p,q}(\mathbb{R}^{n})$ of
negative-order. It is easy to verify that \eqref{eq1.5} and
\eqref{eq1.6} have the same scaling property as the Navier-Stokes
equations (which are equations obtained by putting $d=0$ in
\eqref{eq1.1}--\eqref{eq1.3}), namely, if $(\mathbf{u},F)$ is a
solution of \eqref{eq1.5} and \eqref{eq1.6} with initial data
$(\mathbf{u}_{0}, F_{0})$, then for any $\delta>0$, by letting
\begin{equation}\label{eq1.8}
  \mathbf{u}_{\delta}=\delta \mathbf{u}(\delta x,\delta^{2}t), \quad F_{\delta}=\delta
  F(\delta x,\delta^{2}t),
\end{equation}
we see that $(\mathbf{u}_{\delta}, F_{\delta})$ is a solution of
\eqref{eq1.5} and \eqref{eq1.6} with initial data
$(\delta\mathbf{u}_{0}(\delta x)$, $\delta F_{0}(\delta x))$. The
so-called \textit{critical space} for the equations \eqref{eq1.5}
and \eqref{eq1.6} is a function space of $(\mathbf{u}, F)$ such that
the norm in it is invariant under the scaling \eqref{eq1.8}. The
so-called self-similar solutions are solutions satisfying the
scaling relation $\mathbf{u}(t,x)=\mathbf{u}_{\delta}(t,x)$,
$F(t,x)=F_{\delta}(t,x)$ (for all $\delta>0$, $x\in \mathbb{R}^{n}$
and $t\ge0$). Obviously, if $(\mathbf{u}, F)$ is a self-similar
solution, then we must have $$
  \mathbf{u}_0(\lambda x)=\lambda^{-1}\mathbf{u}_0(x),\ \ F_0(\lambda
  x)=\lambda^{-1}F_0(x).
$$
Such initial data do not belong to any Lebesgue  and Sobolev spaces
due to their strong singularity at $x=0$ as well as slow decay as
$|x|\to\infty$, however, they belong to some homogeneous Besov
spaces with negative order. This is the reason why we study the
problem \eqref{eq1.5}--\eqref{eq1.7} in the critical Besov space of
negative-order. By making use of Chemin-Lerner's time-space
estimates of the heat equation and the Fourier localization
technique, we shall prove that when initial data $(\mathbf{u}_{0},
F_{0})$ belongs to the critical Besov space
$\dot{B}^{-1+n/p}_{p,q}(\mathbb{R}^{n})$ for some suitable $p$ and
$q$, there exists a unique local solution, and this solution is
global when initial data is small enough. As a corollary, we get
existence of self-similar solutions; see Section 3. In order to
figure out the relation between the solution obtained here and the
weak solution studied by Lin and Liu in \cite{LL95}, we shall prove
a stability result, which yields in a direct way that all global
weak solutions with the same initial data must coincide with
solutions obtained here, which is called the weak-strong uniqueness.

Our main results are as follows (for notations we refer the reader
to see Section 2):

\begin{theorem}\label{th1.1} {\rm({Existence})}
Let $n\geq2$, $2\leq p<2n$ and $1\leq r\leq \infty$. Suppose that
$(\mathbf{u}_{0}, F_{0})\in \dot{B}^{-1+n/p}_{p,r}(\mathbb{R}^{n})$
and {\rm $\text{div } \mathbf{u}_{0}=0$}. Then there exists $T>0$
and a unique solution $(\mathbf{u}, F)$ of the Cauchy problem
\eqref{eq1.5}--\eqref{eq1.7} such that
\begin{equation*}\label{eq1.10}
  (\mathbf{u}, F)\in \underset{1<q\leq\infty}{\cap}\mathfrak{L}^{q}(0, T;
  \dot{B}^{-1+n/p+2/q}_{p,r}(\mathbb{R}^{n})).
\end{equation*}
Moreover, there exists $\varepsilon>0$ such that if
$\|(\mathbf{u}_{0},
F_{0})\|_{\dot{B}^{-1+n/p}_{p,r}}\leq\varepsilon$, then the above
assertion holds for $T=\infty$, i.e., the solution $(\mathbf{u}, F)$
is global. Furthermore, if $(\mathbf{u},F)$ and
$(\tilde{\mathbf{u}}, \tilde{F})$ are two solutions of
\eqref{eq1.5}--\eqref{eq1.7} with initial data $(\mathbf{u}_0, F_0)$
and $(\tilde{\mathbf{u}}_0, \tilde{F}_0)$, respectively, {\rm
$\text{div } \mathbf{u}_{0}=0$} and {\rm $\text{div }
\tilde{\mathbf{u}}_{0}=0$}, then there exists a constant $C>0$  such
that
\begin{equation*}
  \|(\mathbf{u}-\tilde{\mathbf{u}}, F-\tilde{F})\|_{\mathfrak{L}^{q}(0,T; \dot{B}^{-1+n/p+2/q}_{p,r})}\le
  C\|(\mathbf{u}_0-\tilde{\mathbf{u}}_0,
  F_0-\tilde{F}_0)\|_{\dot{B}^{-1+n/p}_{p,r}}.
\end{equation*}
\end{theorem}

\noindent\textit{Remark 1.1.} (i) If $(\mathbf{u}_{0}, F_{0})$
belongs to the closure of $\mathcal{S}(\mathbb{R}^{n})$ in
$\dot{B}^{-1+n/p}_{p,q}(\mathbb{R}^{n})$, we actually have
$(\mathbf{u}, F)\in
C([0,T],\dot{B}^{-1+n/p}_{p,q}(\mathbb{R}^{n}))$.

(ii) Since the appearance of the term $F\nabla\mathbf{u}$ in
\eqref{eq1.6}, we need the condition $2\leq p<2n$. For the
Navier-Stokes equations, similar results hold for all $2\leq
p<\infty$.

Now by the standard procedure, based on the uniqueness of solutions,
allows us to deduce the existence of self-similar solutions of
\eqref{eq1.5}--\eqref{eq1.7}.

\begin{corollary}\label{co1.2}
Let $n\geq2$, $n\leq p<2n$. Suppose that $(\mathbf{u}_{0}, F_{0})\in
\dot{B}^{-1+n/p}_{p,\infty}(\mathbb{R}^{n})$ and {\rm $\text{div }
\mathbf{u}_{0}=0$}, and furthermore, $\mathbf{u}_{0}$ and $F_{0}$
are homogeneous functions with degree $-1$, i.e.,
\begin{equation*}
  \mathbf{u}_{0}(x)=\delta\mathbf{u}_{0}(\delta x), \quad F_{0}(x)
  =\delta F_{0}(\delta x).
\end{equation*}
Then the global solution $(\mathbf{u}, F)$ constructed in Theorem
\ref{th1.1} is a self-similar solution.
\end{corollary}

\begin{theorem}\label{th1.3} {\rm({Blow-up criterion})}
Under the hypotheses of Theorem \ref{th1.1}, we denote by $T^{*}$
the maximum existence time. If $T^{*}<\infty$, then for any $2\leq
p<2n$ and $2<q<\infty$ satisfying
$\frac{n}{p}+\frac{2}{q}>\frac{3}{2}$, we have
\begin{equation*}
  \|(\mathbf{u}, F)\|_{L^{q}(0,T^{*}; \dot{B}^{-1+n/p+2/q}_{p,q})}=\infty.
\end{equation*}
\end{theorem}




For initial data $(\mathbf{u}_{0}, F_{0})\in
\dot{B}^{-1+n/p}_{p,q}(\mathbb{R}^{n})$, by Theorem \ref{th1.1},
there exists $T>0$ such that the system \eqref{eq1.5}--\eqref{eq1.7}
has a unique solution $(\mathbf{u}, F)\in
\cap_{1<q\leq\infty}\mathfrak{L}^{q}(0,T;
\dot{B}^{-1+n/p+2/q}_{p,r}(\mathbb{R}^{n}))$. If
$(\mathbf{u}_{0},F_{0})$ is additionally in the space
$L^{2}(\mathbb{R}^{n})$, then it is not difficult to see that
$(\mathbf{u}, F)$ is also a weak solution (in the standard sense,
see \cite{LL95}).  A natural question is the following: Do all weak
solutions coincide with the one we obtained in Theorem \ref{th1.1}?
In order to answer this question, we establish the following
stability theorem.

\begin{theorem}\label{th1.4} {\rm({Stability})}
Let $n\geq 2$, $2\leq p<\infty$, $2<q<\infty$ and
$\frac{n}{p}+\frac{2}{q}>1$. Assume that $(\mathbf{u}_{0}, F_{0})$
and $(\tilde{\mathbf{u}}_{0}, \tilde{F}_{0})$ be two vector fields
in $L^{2}(\mathbb{R}^{n})$ such that {\rm $\text{div }
\mathbf{u}_{0}=0$} and {\rm $\text{div } \tilde{\mathbf{u}}_{0}=0$},
and $(\tilde{\mathbf{u}}, \tilde{F})$ and $(\mathbf{u}, F)$ be two
weak solutions associated with initial data
$(\tilde{\mathbf{u}}_{0}, \tilde{F}_{0})$ and $(\mathbf{u}_{0},
F_{0})$, respectively. Let
$\mathbf{w}=\mathbf{u}-\tilde{\mathbf{u}}$, $E=F-\tilde{F}$. If we
assume further that $(\mathbf{u}, F)\in L^{q}(0,T;
\dot{B}^{-1+n/p+2/q}_{p,q}(\mathbb{R}^{n}))$, then  for any $0<
t\leq T$,
\begin{equation*}
  \|(\mathbf{w}, E)\|_{L^{2}}^{2}+2\int^{t}_{0}\|(\nabla\mathbf{w}, \nabla E)\|_{L^{2}}^{2}d\tau
  \leq \|(\mathbf{w}_{0}, E_{0})\|_{L^{2}}^{2}\times\exp\Big(C\int^{t}_{0}\|(\mathbf{u},
  F)\|_{\dot{B}^{-1+n/p+2/q}_{p,q}}^{q}d\tau\Big),
\end{equation*}
where $\mathbf{w}_{0}=\mathbf{u}_{0}-\tilde{\mathbf{u}}_{0}$ and
$E_{0}=F_{0}-\tilde{F}_{0}$, and $C>0$ is a constant.
\end{theorem}

It is clear that if $\mathbf{u}_{0}=\tilde{\mathbf{u}}_{0}$ and
$F_{0}=\tilde{F}_{0}$, then Theorem \ref{th1.4} implies that
$\mathbf{w}=E=0$, i.e., $\mathbf{u}=\tilde{\mathbf{u}}$ and
$F=\tilde{F}$. Hence, we have the following weak-strong uniqueness
result for the system \eqref{eq1.5}--\eqref{eq1.7}.

\begin{corollary}\label{co1.5}{\rm({Weak-strong uniqueness})}
Let $n\geq 2$, $2\leq p<\infty$, $2<q<\infty$ and
$\frac{n}{p}+\frac{2}{q}>1$. Assume that $(\mathbf{u}_{0},F_{0})\in
L^{2}(\mathbb{R}^{n})$ {\rm($\text{div }\mathbf{u}_{0}=0$)}, and
$(\mathbf{u}, F)$ be a weak solution of the problem
\eqref{eq1.5}--\eqref{eq1.7} with initial data
$(\mathbf{u}_{0},F_{0})$. If we assume further that $(\mathbf{u},
F)\in L^{q}(0,T; \dot{B}^{-1+n/p+2/q}_{p,q}(\mathbb{R}^{n}))$, then
all weak solutions associated with initial data
$(\mathbf{u}_{0},F_{0})$ must coincide with $(\mathbf{u}, F)$ on the
time interval $[0,T)$.
\end{corollary}

\noindent\textit{Remark 1.2.} For initial data $(\mathbf{u}_{0},
F_{0})\in \dot{B}^{-1+n/p}_{p,r}(\mathbb{R}^{n})$, when $2\leq
p<2n$, $1\leq r<\infty$ such that $\frac{n}{p}+\frac{2}{r}>1$, by
Theorem \ref{th1.1}, there exists $2<q<\infty$ such that
$\frac{n}{p}+\frac{2}{q}>1$, and the system
\eqref{eq1.5}--\eqref{eq1.7} has a unique solution $(\mathbf{u},
F)\in L^{q}(0,T; \dot{B}^{-1+n/p+2/q}_{p,q}(\mathbb{R}^{n}))$.
Hence, in this case, by Corollary \ref{co1.5}, weak-strong
uniqueness holds for the system \eqref{eq1.5}--\eqref{eq1.7}.

\noindent\textit{Organization of the paper.} In Section 2, we recall
some basic facts about Littlewood-Paley decomposition and Besov
spaces. In Section 3, we present the proof of Theorem \ref{th1.1},
which yields existence of global self-similar solutions. In Section
4, we prove Theorem \ref{th1.3}. Section 5 is devoted to the proof
of Theorem \ref{th1.4}.

\section{\bf Preliminaries}

We first recall some basic notions and preliminary results used in
the proof of our main results. Let $\mathcal{S}(\mathbb{R}^{n})$ be
the Schwartz space and $\mathcal{S}'(\mathbb{R}^{n})$ be its dual.
Given $f\in\mathcal{S}(\mathbb{R}^{n})$, the Fourier transform of
it, $\mathcal{F}(f)=\widehat{f}$, is defined by
$$
  \mathcal{F}(f)(\xi)=\widehat{f}(\xi)=\frac{1}{(2\pi)^{n/2}}\int_{\mathbb{R}^{n}}f(x)e^{-ix\cdot\xi}dx.
$$
Let $\mathcal{D}_{1}=\{\xi\in\mathbb{R}^{n},\
|\xi|\leq\frac{4}{3}\}$ and
$\mathcal{D}_{2}=\{\xi\in\mathbb{R}^{n},\ \frac{3}{4}\leq
|\xi|\leq\frac{8}{3}\}$. Choose two non-negative functions $\phi,
\psi\in\mathcal{S}(\mathbb{R}^{n})$ supported, respectively, in
$\mathcal{D}_{1}$ and $\mathcal{D}_{2}$ such that
\begin{align*}
  \psi(\xi)+\sum_{j\geq0}\phi(2^{-j}\xi)=1, \ \ \xi\in\mathbb{R}^{n},\\
  \sum_{j\in\mathbb{Z}}\phi(2^{-j}\xi)=1, \ \
  \xi\in\mathbb{R}^{n}\backslash\{0\}.
\end{align*}
We denote $\phi_{j}(\xi)=\phi(2^{-j}\xi)$, $h=\mathcal{F}^{-1}\phi$
and $\tilde{h}=\mathcal{F}^{-1}\psi$, where $\mathcal{F}^{-1}$ is
the inverse Fourier transform. Then the dyadic blocks $\Delta_{j}$
and $S_{j}$ can be defined as follows:
\begin{align*}
  \Delta_{j}f=\phi(2^{-j}D)f=2^{jn}\int_{\mathbb{R}^{n}}h(2^{j}y)f(x-y)dy,\\
  S_{j}f=\psi(2^{-j}D)f=2^{jn}\int_{\mathbb{R}^{n}}\tilde{h}(2^{j}y)f(x-y)dy.
\end{align*}
Here $D=(D_1, D_2, \cdots, D_n)$ and $D_j=i^{-1}\partial_{x_j}$
($i^{2}=-1$). The set $\{\Delta_{j}, S_{j}\}_{j\in\mathbb{Z}}$ is
called the Littlewood-Paley decomposition. Formally,
$\Delta_{j}=S_{j}-S_{j-1}$ is a frequency projection to the annulus
$\{|\xi|\sim 2^{j}\}$, and $S_{j}=\sum_{k\leq j-1}\Delta_{k}$ is a
frequency projection to the ball $\{|\xi|\leq 2^{j}\}$. For more
details, please reader to \cite{C98} and \cite{L02}.  Let
$\mathcal{Z}(\mathbb{R}^{n})=\big\{f\in \mathcal{S}(\mathbb{R}^{n}):
\ \ \partial^{\alpha}\widehat{f}(0)=0, \
\forall\alpha\in(\mathbb{N}\cup\{0\})^{n}\big\}$, and denote by
$\mathcal{Z}'(\mathbb{R}^{n})$ the dual of it.

\begin{definition}\label{de2.1}
Let $s\in\mathbb{R}$, $(p,r)\in[1, \infty]\times[1, \infty]$, the
homogeneous Besov space $\dot{B}^{s}_{p,r}(\mathbb{R}^{n})$ is
defined by
\begin{equation*}
  \dot{B}^{s}_{p,r}(\mathbb{R}^{n})=\Big\{f\in \mathcal{Z}'(\mathbb{R}^{n}):\ \
  \|f\|_{\dot{B}^{s}_{p,r}}<\infty\Big\},
\end{equation*}
where
\begin{equation*}
  \|f\|_{\dot{B}^{s}_{p,r}}=
\begin{cases}
  \Big(\sum_{j\in\mathbb{Z}}2^{jsr}\|\Delta_{j}f\|_{L^{p}}^{r}\Big)^{1/r}\
  \ \text{for}\ \ 1\leq r<\infty,\\
  \sup_{j\in\mathbb{Z}}2^{js}\|\Delta_{j}f\|_{L^{p}}\ \ \ \ \ \ \ \ \ \text{for}\
  \ r=\infty.
\end{cases}
\end{equation*}
\end{definition}

\noindent\textit{Remark 2.1.} The above definition does not depend
on the choice of the couple $(\phi, \psi)$. Recall that if either
$s<\frac{n}{p}$ or $s=\frac{n}{p}$ and $q=1$, then
$(\dot{B}^{s}_{p,q}(\mathbb{R}^{n}),\|\cdot\|_{\dot{B}^{s}_{p,q}})$
is a Banach space.

Let us now state some basic properties for the homogeneous Besov
spaces.

\begin{lemma}\label{le2.2} {\rm (Bernstein's inequality \cite{C98})}
Let $k\in\mathbb{Z}^{+}$. There exists a constant $C$ independent of
$f$ and $j$ such that for all $1\leq p\leq q\leq\infty$, the
following estimate holds:
\begin{equation}\label{eq2.1}
  \supp \widehat{f}\subset\{|\xi|\leq2^{j}\}\ \ \Longrightarrow\ \
  \sup_{|\alpha|=k}\|\partial^{\alpha}f\|_{L^{q}}\leq
  C2^{jk+jn(1/p-1/q)}\|f\|_{L^{p}}.
\end{equation}
\end{lemma}

\begin{lemma}\label{le2.3} {\rm (\cite{C98})}
Let $k\in\mathbb{Z}^{+}$ and $|\alpha|=k$ for multi-index $\alpha$.
There exists a constant $C_{k}$ depending only on $k$ such that for
all $s\in \mathbb{R}$ and  $1\leq p,r\leq\infty$, the following
estimate holds:
\begin{equation}\label{eq2.2}
  C_{k}^{-1}\|\partial^{\alpha}f\|_{\dot{B}^{s}_{p,r}}\leq \|f\|_{\dot{B}^{s+k}_{p,q}}\leq C_{k}\|\partial^{\alpha}f\|_{\dot{B}^{s}_{p,r}}.
\end{equation}
\end{lemma}

We now recall the definition of the Chemin-Lerner space
$\mathfrak{L}^{r}(0,T; \dot{B}^{s}_{p,q}(\mathbb{R}^{n}))$:

\begin{definition}\label{de2.4}
Let $s\in \mathbb{R}$, $1\leq p, q, r\leq\infty$, and let $T>0$ be a
fixed number, the space $\mathfrak{L}^{q}(0,T;
\dot{B}^{s}_{p,r}(\mathbb{R}^{n}))$ is defined by
\begin{equation*}
  \mathfrak{L}^{q}(0,T; \dot{B}^{s}_{p,r}(\mathbb{R}^{n})):=\Big\{f\in
  \mathcal{S}'((0,T), \mathcal{Z}'(\mathbb{R}^{n})):\ \
  \|f\|_{\mathfrak{L}^{q}(0,T;
  \dot{B}^{s}_{p,r}(\mathbb{R}^{n}))}<\infty\Big\},
\end{equation*}
where
\begin{equation*}
  \|f\|_{\mathfrak{L}^{q}(0,T; \dot{B}^{s}_{p,r})}=\Big(\sum_{j\in\mathbb{Z}}2^{jsr}\|\Delta_{j}f\|_{L^{q}(0,T;
  L^{p})}^{r}\Big)^{1/r}.
\end{equation*}
\end{definition}

\noindent\textit{Remarks 2.2.} (i) We define the usual space
$L^{q}(0,T; \dot{B}^{s}_{p,r}(\mathbb{R}^{n}))$ equipped with the
norm
\begin{equation*}
  \|f\|_{L^{q}(0,T; \dot{B}^{s}_{p,r})}=\Big(\int_{0}^{T}\Big(\sum_{j\in\mathbb{Z}}2^{jsr}\|\Delta_{j}f\|_{
  L^{p}}^{r}\Big)^{q/r}dt\Big)^{1/q}.
\end{equation*}

(ii) By  Minkowski's inequality, it is readily to verify that
\begin{align}\label{eq2.3}
  \|f\|_{\mathfrak{L}^{q}(0,T; \dot{B}^{s}_{p,r})}\leq\|f\|_{L^{q}(0,T;
  \dot{B}^{s}_{p,r})} \ \ \  \text{if}\ \ \  q\leq r,\\
\label{eq2.4}
  \|f\|_{L^{q}(0,T; \dot{B}^{s}_{p,r})}\leq \|f\|_{\mathfrak{L}^{q}(0,T;
  \dot{B}^{s}_{p,r})} \ \ \ \text{if} \ \ \  r\leq q.
\end{align}

The following product between functions will enable us to estimate
nonlinear terms appeared in \eqref{eq1.5} and \eqref{eq1.6}.

\begin{lemma}\label{le2.5} {\rm (\cite{D05}, \cite{RS96})}
Let $1\leq p$, $q$, $r$, $q_{1}$, $q_{2}\leq \infty$, $s_{1}$,
$s_{2}<\frac{n}{p}$, $s_{1}+s_{2}>0$ and
$\frac{1}{q}=\frac{1}{q_{1}}+\frac{1}{q_{2}}$. Then there exists a
positive constant $C$ depending only on $s_{1}, s_{2}, p, q,
r,q_{1}, q_{2}$ and $n$ such that
\begin{equation}\label{eq2.5}
  \|fg\|_{\mathfrak{L}^{q}(0,T; \dot{B}^{s_{1}+s_{2}-\frac{n}{p}}_{p,
  r})}\leq C\|f\|_{\mathfrak{L}^{q_{1}}(0,T; \dot{B}^{s_{1}}_{p,
  r})}\|g\|_{\mathfrak{L}^{q_{2}}(0,T; \dot{B}^{s_{2}}_{p,
  r})}.
\end{equation}
\end{lemma}

\noindent\textit{Notations:} The product of Banach spaces
$\mathcal{X}\times\mathcal{Y}$ will be equipped with the usual norm
$\|(f,g)\|_{\mathcal{X}\times\mathcal{Y}}=\|f\|_{\mathcal{X}}+\|g\|_{\mathcal{Y}}$,
and if $\mathcal{X}=\mathcal{Y}$, we use $\|(f,g)\|_{\mathcal{X}}$
to denote by $\|(f,g)\|_{\mathcal{X}\times\mathcal{X}}$. For two
$n\times n$ matrixes $A=(a_{ij})_{i,j=1}^{n}$ and
$B=(b_{ij})_{i,j=1}^{n}$, we denote
$A:B=\sum^{n}_{i,j=1}a_{ij}b_{ij}$. Throughout the paper, $C$ stands
for a generic constant, and its value may change from line to line.

\section{Local and global existence of solution}
In this section we prove Theorem \ref{th1.1}. Let $p$ and $r$ be as
in Theorem 1.1, i.e., $2\leq p<2n$ and $1\leq r\leq\infty$, and let
$1<q\leq\infty$. We choose a number $2<q_1\leq2q$ such that
$\frac{2}{q_1}+\frac{n}{p}>\frac{3}{2}$. For a constant $T>0$ to be
specified later, we denote $\mathcal{X}_{T}=\mathfrak{L}^{q_1}(0,T;
\dot{B}^{-1+n/p+2/q_1}_{p,r}(\mathbb{R}^{n}))$. In order to
establish the desired estimates in $\mathcal{X}_{T}$, let us recall
the solvability of the Cauchy problem of the heat equation:
\begin{equation}\label{eq3.1}
\begin{cases}
  \partial_{t} u-\Delta u= f(x,t), \ \
  x\in\mathbb{R}^{n}, \ t>0,\\
  u(x,0)=u_{0}(x), \ \ x\in\mathbb{R}^{n}.
\end{cases}
\end{equation}

\begin{proposition}\label{pro3.1} {\rm (\cite{D05})}
Let $s\in \mathbb{R}$ and $1\leq p,q_1,r\leq\infty$, and let $T>0$
be a real number. Assume that $u_{0}\in
\dot{B}^{s}_{p,r}(\mathbb{R}^{n})$ and $f\in\mathfrak{L}^{q_1}(0,T;
\dot{B}^{s+2/q_1-2}_{p,r}(\mathbb{R}^{n}))$. Then the Cauchy problem
\eqref{eq3.1} has a unique solution
\begin{equation*}
  u\in\underset{q_1\leq q\leq\infty}{\cap}\mathfrak{L}^{q}(0,T;\dot{B}^{s+2/q}_{p,r}(\mathbb{R}^{n})).
\end{equation*}
Moreover, there exists a constant $C>0$ depending only on $n$ such
that for any $q_1\leq q\leq \infty$,
\begin{equation}\label{eq3.2}
  \|u\|_{\mathfrak{L}^{q}(0,T; \dot{B}^{s+2/q}_{p,r})}\leq
  C\Big(\|u_{0}\|_{\dot{B}^{s}_{p,r}}+\|f\|_{\mathfrak{L}^{q_1}(0,T;
  \dot{B}^{s+2/q_1-2}_{p,r})}\Big).
\end{equation}
Besides, if $u_{0}$ belongs to the closure of
$\mathcal{S}(\mathbb{R}^n)$ in $\dot{B}^{s}_{p,r}(\mathbb{R}^{n})$,
then $u\in C([0,T), \dot{B}^{s}_{p,r}(\mathbb{R}^{n}))$.

\end{proposition}

We also recall an existence and uniqueness result for an abstract
operator equation in a generic Banach space. For the proof we refer
the reader to see Lemari\'{e}-Rieusset \cite{L02}.

\begin{proposition}\label{pro3.2} {\em (\cite{L02})}
Let $\mathcal{X}$ be a Banach space and
$\mathbf{B}:\mathcal{X}\times\mathcal{X}\rightarrow\mathcal{X}$ is a
bilinear bounded operator, $\|\cdot\|_{\mathcal{X}}$ being the
$\mathcal{X}$-norm. Assume that for any $u_{1},u_{2}\in
\mathcal{X}$, we have
\begin{align*}
  \|\mathbf{B}(u_{1},u_{2})\|_{\mathcal{X}}\leq C_{0}
  \|u_{1}\|_{\mathcal{X}}\|u_{2}\|_{\mathcal{X}}.
\end{align*}
Then for any $y\in \mathcal{X}$ such that $\|y\|_{\mathcal{X}}\leq
\varepsilon<\frac{1}{4C_{0}}$, the equation $u=y+\mathbf{B}(u,u)$
has a solution $u$ in $\mathcal{X}$. Moreover, this solution is the
only one such that $\|u\|_{\mathcal{X}}\leq 2\varepsilon$, and
depends continuously on $y$ in the following sense: if
$\|\widetilde{y}\|_{\mathcal{X}}\leq \varepsilon$,
$\widetilde{u}=\widetilde{y}+\mathbf{B}(\widetilde{u},\widetilde{u})$
and $\|\widetilde{u}\|_{\mathcal{X}}\leq 2\varepsilon$, then
\begin{align*}
  \|u-\widetilde{u}\|_{\mathcal{X}}\leq \frac{1}{1-4\varepsilon
  C_{0}}\|y-\widetilde{y}\|_{\mathcal{X}}.
\end{align*}
\end{proposition}

Now for given $(\mathbf{u},F)\in \mathcal{X}_{T}$, we define by
$\mathcal{G}(\mathbf{u},F)=(\bar{\mathbf{u}},\bar{F})$, where
$(\bar{\mathbf{u}}, \bar{F})$ is a solution of the following linear
equations:
\begin{align}\label{eq3.3}
  &\partial_{t} \bar{\mathbf{u}}-\Delta \bar{\mathbf{u}}=-\mathbb{P}\mathbf{u}\cdot\nabla\mathbf{u}
  -\mathbb{P}\diva(F^{T}F),\\
\label{eq3.4}
  &\partial_{t} \bar{F}-\Delta \bar{F}=-\mathbf{u}\cdot\nabla F-F\nabla
  \mathbf{u},\\
\label{eq3.5}
  &(\bar{\mathbf{u}},\bar{F})|_{t=0}=(\mathbf{u}_{0},F_{0}).
\end{align}

\begin{proposition}\label{pro3.3} Let $(\mathbf{u}, F)\in
\mathcal{X}_{T}$. Then we have $(\bar{\mathbf{u}}$, $\bar{F})\in
\mathcal{X}_{T}$. In addition, the following estimates hold:
\begin{align}\label{eq3.6}
  &\|\bar{\mathbf{u}}\|_{\mathcal{X}_{T}}\leq
  \|e^{t\Delta}\mathbf{u}_{0}\|_{\mathcal{X}_{T}}+C(\|\mathbf{u}\|_{\mathcal{X}_{T}}^{2}+\|F\|_{\mathcal{X}_{T}}^{2}),\\
\label{eq3.7}
  &\|\bar{F}\|_{\mathcal{X}_{T}}\leq
  \|e^{t\Delta}F_{0}\|_{\mathcal{X}_{T}}+C\|\mathbf{u}\|_{\mathcal{X}_{T}}\|F\|_{\mathcal{X}_{T}}.
\end{align}
\end{proposition}
\begin{proof}
We prove only the results for $\bar{\mathbf{u}}$, it can be done
analogous for $\bar{F}$. By the Duhamel principle, \eqref{eq3.3} can
be transformed into the following equivalent integral equations:
$$
  \bar{\mathbf{u}}(t)=e^{t\Delta}\mathbf{u}_{0}-\int_{0}^{t}e^{(t-\tau)\Delta}\mathbb{P}(\mathbf{u}\cdot\nabla\mathbf{u})(\tau)d\tau
  -\int_{0}^{t}e^{(t-\tau)\Delta}\mathbb{P}\diva(F^{T}F)(\tau)d\tau.
$$
Since we have assumed $2\leq p<2n$, $2<q_1<\infty$ and
$\frac{n}{p}+\frac{2}{q_1}>\frac{3}{2}$, we can apply Lemma
\ref{le2.5} by choosing $s_{1}=-1+\frac{n}{p}+\frac{2}{q_1}$ and
$s_{2}=-2+\frac{n}{p}+\frac{2}{q_1}$ and Lemma \ref{le2.3} to obtain
that
\begin{align}\label{eq3.8}
  \|\mathbf{u}\cdot\nabla\mathbf{u}\|_{\mathfrak{L}^{q_1/2}(0,T;
  \dot{B}^{-3+n/p+4/q_1}_{p,r})}
  &\leq C\|\mathbf{u}\|_{\mathfrak{L}^{q_1}(0,T;
  \dot{B}^{-1+n/p+2/q_1}_{p,r})}\|\nabla\mathbf{u}\|_{\mathfrak{L}^{q_1}(0,T;
  \dot{B}^{-2+n/p+2/q_1}_{p,r})}\nonumber\\
  &\leq C\|\mathbf{u}\|_{\mathfrak{L}^{q_1}(0,T;
  \dot{B}^{-1+n/p+2/q_1}_{p,r})}^{2}.
\end{align}
Similarly, by Lemmas \ref{le2.3} and \ref{le2.5},
\begin{align}\label{eq3.9}
  \|\text{div}(F^{T}F)\|_{\mathfrak{L}^{q_1/2}(0,T; \dot{B}^{-3+n/p+4/q_1}_{p,r})}
  &\leq C\|F^{T}F\|_{\mathfrak{L}^{q_1/2}(0,T; \dot{B}^{-2+n/p+4/q_1}_{p,r})}\nonumber\\
  &\leq C\|F\|_{\mathfrak{L}^{q_1}(0,T;
  \dot{B}^{-1+n/p+2/q_1}_{p,r})}^{2}.
\end{align}
Hence, from Proposition \ref{pro3.1} we know $\mathbf{u}\in
\mathcal{X}_{T}$. Moreover, by using the boundedness of $\mathbb{P}$
in the homogeneous Besov spaces, \eqref{eq3.8} and \eqref{eq3.9}, we
get
\begin{align}\label{eq3.10}
  \|\bar{\mathbf{u}}\|_{\mathcal{X}_{T}}&\!\leq\!
  \|e^{t\Delta}\mathbf{u}_{0}\|_{\mathcal{X}_{T}}\!+\!C\|\mathbf{u}\cdot\nabla\mathbf{u}\|_{\mathfrak{L}^{q_1/2}(0,T;
  \dot{B}^{-3+n/p+4/q_1}_{p,r})}\!+\!C\|\text{div}(F^{T}F)\|_{\mathfrak{L}^{q_1/2}(0,T;
  \dot{B}^{-3+n/p+4/q_1}_{p,r})}\nonumber\\
  &\leq\|e^{t\Delta}\mathbf{u}_{0}\|_{\mathcal{X}_{T}}+C(\|\mathbf{u}\|_{\mathcal{X}_{T}}^{2}+\|F\|_{\mathcal{X}_{T}}^{2}).
\end{align}
This proves Proposition \ref{pro3.3}.
\end{proof}

The Proposition \ref{pro3.3} implies that $\mathcal{G}$ is
well-defined and maps $\mathcal{X}_{T}$ into itself. Moreover, from
\eqref{eq3.6} and  \eqref{eq3.7}, we know there exists a constant
$C_0>0$ such that for all $(\mathbf{u},F)\in \mathcal{X}_{T}$ and
$(\mathbf{\bar{u}},\bar{F})= \mathcal{G}(\mathbf{u},F)$, we have the
following estimate:
\begin{align}\label{eq3.11}
  \|(\mathbf{\bar{u}},\bar{F})\|_{\mathcal{X}_{T}}\leq
  \|(e^{t\Delta}\mathbf{u}_{0},e^{t\Delta} F_{0})\|_{\mathcal{X}_{T}}
  + C_{0}\|(\mathbf{u},F)\|_{\mathcal{X}_{T}}^2.
\end{align}
\textbf{Case 1.} (The small initial data). Taking $T=\infty$ and
denoting $\mathcal{X}=\mathcal{X}_{\infty}$. From Proposition
\ref{pro3.1}, there exists a constant $C_{1}$ such that we can
rewrite \eqref{eq3.11} as follows:
\begin{align}\label{eq3.12}
  \|(\mathbf{\bar{u}},\bar{F})\|_{\mathcal{X}}\leq
  C_{1}\|(\mathbf{u}_{0}, F_{0})\|_{\dot{B}^{-1+n/p}_{p,r}}
  +C_{0}\|(\mathbf{u},F)\|_{\mathcal{X}}^2.
\end{align}
Now if we choose $\varepsilon>0$ sufficiently small such that
$C_{1}\|(\mathbf{u}_{0},
F_{0})\|_{\dot{B}^{-1+n/p}_{p,r}}\leq\varepsilon<\frac{1}{4C_{0}}$,
that is to say,  $\|(\mathbf{u}_{0},
F_{0})\|_{\dot{B}^{-1+n/p}_{p,r}}\leq\frac{\varepsilon}{C_{1}}<\frac{1}{4C_{0}C_{1}}$,
then by Proposition \ref{pro3.2}, the system
\eqref{eq1.5}--\eqref{eq1.7} has a global solution.

\textbf{Case 2.} (The large initial data). In this case we shall use
the Fourier localization technique to obtain existence of local
solution of the problem \eqref{eq1.5}--\eqref{eq1.7}. To this end,
we split $\mathbf{u}_{0}=\mathbf{u}_{01}+\mathbf{u}_{02}$ such that
$\widehat{\mathbf{u}}_{0}(\xi)=\widehat{\mathbf{u}}_{0}1_{\{|\xi|>2^{N}\}}+\widehat{\mathbf{u}}_{0}1_{\{|\xi|\leq2^{N}\}}:
=\widehat{\mathbf{u}_{01}}+\widehat{\mathbf{u}_{02}}$, where
$1_{\mathcal{D}}$ represents the characteristic function on the
domain $\mathcal{D}$. Similarly, we split $F_{0}=F_{01}+F_{02}$.
Since $2\leq p<2n$, by using the properties of the Besov spaces,
there exists $N\in\mathbb{Z}^{+}$ such that
$C_{1}\|(\mathbf{u}_{01},F_{01})\|_{\dot{B}^{-1+n/p}_{p,r}}\leq
\frac{1}{2}\varepsilon$, we see that
\begin{equation}\label{eq3.13}
  \|(e^{t\Delta}\mathbf{u}_{0},
  e^{t\Delta}F_{0})\|_{\mathcal{X}_{T}}\leq
  \frac{1}{2}\varepsilon+\|(e^{t\Delta}\mathbf{u}_{02},
  e^{t\Delta}F_{02})\|_{\mathcal{X}_{T}}.
\end{equation}
Applying the Bernstein's inequality, there exists a constant $C_{2}$
such that
\begin{align*}
  \|(e^{t\Delta}\mathbf{u}_{02},
  e^{t\Delta}F_{02})\|_{\mathcal{X}_{T}}&=\|(e^{t\Delta}\mathbf{u}_{02},
  e^{t\Delta}F_{02})\|_{\mathfrak{L}^{q_1}(0,T;
  \dot{B}^{-1+n/p+2/q_1}_{p,r})}\\
  &\leq C2^{(2N)/q_1}\|(e^{t\Delta}\mathbf{u}_{02},
  e^{t\Delta}F_{02})\|_{\mathfrak{L}^{q_1}(0,T;
  \dot{B}^{-1+n/p}_{p,r})}\\&\leq C_{2}2^{(2N)/q_1}T^{1/q_1}\|(\mathbf{u}_{0},
  F_{0})\|_{\dot{B}^{-1+n/p}_{p,r}}.
\end{align*}
Hence, if we choose $T$ small enough so that
$C_{2}2^{(2N)/q_1}T^{1/q_1}\|(\mathbf{u}_{0},F_{0})\|_{\dot{B}^{-1+n/p}_{p,r}}\leq\frac{1}{2}\varepsilon$,
i.e.,
\begin{equation}\label{eq3.14}
  T\leq
  \Big(\frac{\varepsilon}{C_{2}2^{1+(2N)/q_1}\|(\mathbf{u}_{0},F_{0})\|_{\dot{B}^{-1+n/p}_{p,r}}}\Big)^{q_1},
\end{equation}
then we have $\|(e^{t\Delta}\mathbf{u}_{02},
e^{t\Delta}F_{02})\|_{\mathcal{X}_{T}}\leq \frac{\varepsilon}{2}$.
This result together with  \eqref{eq3.13} yield the fact that for
such $T$ defined by \eqref{eq3.14}, we have
$\|(e^{t\Delta}\mathbf{u}_{0},
e^{t\Delta}F_{0})\|_{\mathcal{X}_{T}}\leq\varepsilon$. By applying
Proposition \ref{pro3.2} again, there exists a local solution to the
system \eqref{eq1.5}--\eqref{eq1.7}.

If $(\mathbf{u},F)\in \mathcal{X}_{T}$ is a solution of the system
\eqref{eq1.5}--\eqref{eq1.7}, then one can proceed the same way as
the proof of Proposition \ref{pro3.3} to obtain that
$$
  \mathbb{P}\mathbf{u}\cdot\nabla\mathbf{u},\
  \mathbb{P}\diva(F^{T}F),\
  \mathbf{u}\cdot\nabla F,\ F\nabla\mathbf{u}\in  \mathfrak{L}^{q_1/2}(0,T;
  \dot{B}^{-3+n/p+4/q_1}_{p,r}(\mathbb{R}^{n})).
$$
Hence, for any $\frac{q_1}{2}\leq q\leq\infty$, we have
$$
  (\mathbf{u}, F)\in \mathfrak{L}^{q}(0,T;
  \dot{B}^{-1+n/p+2/q}_{p,r}(\mathbb{R}^{n})).
$$
Moreover, if $(\mathbf{u}_0,F_0)$ belongs to the closure of
$\mathcal{S}(\mathbb{R}^n)$ in
$\dot{B}^{-1+n/p}_{p,r}(\mathbb{R}^{n})$, then we have $(\mathbf{u},
F)\in C([0,T], \dot{B}^{-1+n/p}_{p,r}(\mathbb{R}^{n}))$.

Finally, we consider the uniqueness of solution. Note that in
Proposition \ref{pro3.2} we obtained only a partial answer to the
uniqueness problem of solution, i.e., in the closed ball
$B_{2\varepsilon}$, the solution of \eqref{eq1.5}--\eqref{eq1.7} is
unique. Now we intend to get rid of this restrictive condition.

Let $(\mathbf{u},F)$ and $(\tilde{\mathbf{u}},\tilde{F})$ be two
solutions of \eqref{eq1.5}--\eqref{eq1.7} in $\mathcal{X}_{T}$
associated with initial data $(\mathbf{u}_{0},F_{0})$ and
$(\tilde{\mathbf{u}}_{0}, \tilde{F}_{0})$, respectively. Set
$\mathbf{w}=\mathbf{u}-\tilde{\mathbf{u}}$ and $E=F-\tilde{F}$. Then
$(\mathbf{w},E)$ satisfies the following equations:
\begin{equation*}
\begin{cases}
  &\partial_{t} \mathbf{w}-\Delta
  \mathbf{w}+\mathbf{w}\cdot\nabla\mathbf{u}+\tilde{\mathbf{u}}\cdot\nabla\mathbf{w}+\nabla\cdot(E^{T}F)+\nabla\cdot(\tilde{F}^{T}E)=0,\\
  &\partial_{t} E-\Delta
  E+\mathbf{w}\cdot\nabla F+\tilde{\mathbf{u}}\cdot\nabla E+E\nabla\mathbf{u}+\tilde{F}\nabla\mathbf{w}=0,\\
  &\mathbf{w}(x,0)=\mathbf{w}_{0}(x)=\mathbf{u}_{0}(x)-\tilde{\mathbf{u}}_{0}(x),\
  \ E(x,0)=E_{0}(x)=F_{0}(x)-\tilde{F}_{0}(x).
\end{cases}
\end{equation*}
As the proof of Proposition \ref{pro3.3}, we can prove that
\begin{align*}
  \|\mathbf{w}\cdot\nabla\mathbf{u}&+\tilde{\mathbf{u}}\cdot\nabla\mathbf{w}+\nabla\cdot(E^{T}F)+\nabla\cdot(\tilde{F}^{T}E)\|_{\mathcal{X}_{T}}\\
  &\leq
  C\big(\|\mathbf{u}\|_{\mathcal{X}_{T}}+\|\tilde{\mathbf{u}}\|_{\mathcal{X}_{T}}+\|F\|_{\mathcal{X}_{T}}+\|\tilde{F}\|_{\mathcal{X}_{T}}\big)
  \|(\mathbf{w},E)\|_{\mathcal{X}_{T}}
\end{align*}
and
\begin{align*}
  \|\mathbf{w}\cdot\nabla F&+\tilde{\mathbf{u}}\cdot\nabla E+E\nabla\mathbf{u}+\tilde{F}\nabla\mathbf{w}\|_{\mathcal{X}_{T}}\\
  &\leq
  C\big(\|\mathbf{u}\|_{\mathcal{X}_{T}}+\|\tilde{\mathbf{u}}\|_{\mathcal{X}_{T}}+\|F\|_{\mathcal{X}_{T}}+\|\tilde{F}\|_{\mathcal{X}_{T}}\big)
  \|(\mathbf{w},E)\|_{\mathcal{X}_{T}}.
\end{align*}
Hence, by Proposition \ref{pro3.1}, we get
\begin{align*}
  \|(\mathbf{w},E )\|_{\mathcal{X}_{T}}&\leq C_{1}\|(\mathbf{w}_{0},
  E_{0})\|_{\dot{B}^{-1+n/p}_{p,r}}\\&+ C_{0}
  \big(\|\mathbf{u}\|_{\mathcal{X}_{T}}+\|\tilde{\mathbf{u}}\|_{\mathcal{X}_{T}}+\|F\|_{\mathcal{X}_{T}}+\|\tilde{F}\|_{\mathcal{X}_{T}}\big)
  \|(\mathbf{w},E)\|_{\mathcal{X}_{T}}.
\end{align*}
Denoting
$M(T):=C_{0}\big(\|\mathbf{u}\|_{\mathcal{X}_{T}}+\|\tilde{\mathbf{u}}\|_{\mathcal{X}_{T}}
+\|F\|_{\mathcal{X}_{T}}+\|\tilde{F}\|_{\mathcal{X}_{T}}\big)$. By
the Lebesgue dominated convergence theorem, we know that $M(T)$ is a
continuous nondecreasing function vanishing at zero. Hence, if we
choose $T_{1}$ sufficiently small such that
$M(T_{1})\leq\frac{1}{2}$, then
\begin{equation}\label{eq3.15}
  \|(\mathbf{w},E )\|_{\mathcal{X}_{T}}\leq 2C_{1}\|(\mathbf{w}_{0},
  E_{0})\|_{\dot{B}^{-1+n/p}_{p,r}}.
\end{equation}
Repeating the above procedure to the interval $[0,T_{1})$, $[T_{1},
2T_{1})$, $\ldots$ enables us to conclude that there exists a
constant $C$ such that
\begin{equation}\label{eq3.16}
  \|(\mathbf{w},E )\|_{\mathcal{X}_{T}}\leq C\|(\mathbf{w}_{0},
  E_{0})\|_{\dot{B}^{-1+n/p}_{p,r}}.
\end{equation}
This implies the uniqueness result immediately.

\section{The proof of Theorem \ref{th1.3}}

Let $2\leq p<2n$, $1\leq r\leq\infty$, $2<q<\infty$ such that
$\frac{n}{p}+\frac{2}{q}>\frac{3}{2}$. Assume that
$\|(\mathbf{u},F)\|_{\mathfrak{L}^{q}(0,T;
\dot{B}^{-1+n/p+2/q}_{p,r})}<\infty$. By the embedding relation
\eqref{eq2.4} and the proof of Theorem \ref{th1.1} we see that
\begin{align*}
  \|(\mathbf{u},&F)\|_{L^{\infty}(0,T;\dot{B}^{-1+n/p}_{p,r})}\leq\|(\mathbf{u},F)\|_{\mathfrak{L}^{\infty}(0,T;\dot{B}^{-1+n/p}_{p,r})}\\
  &\leq
  C_{1}\|(\mathbf{u}_{0}, F_{0})\|_{\dot{B}^{-1+n/p}_{p,r}}+C_{0}\|(\mathbf{u},F)\|_{\mathfrak{L}^{q}(0,T;
 \dot{B}^{-1+n/p+2/q}_{p,r})}^{2}=M<\infty.
\end{align*} It suffices to prove that if $\|(\mathbf{u},F)\|_{\mathfrak{L}^{q}(0,T;
\dot{B}^{-1+n/p+2/q}_{p,r})}<\infty$, then $T^{*}>T$. In other
words, if $T^{*}<\infty$, then
$\|(\mathbf{u},F)\|_{\mathfrak{L}^{q}(0,T;
\dot{B}^{-1+n/p+2/q}_{p,r})}=\infty$. To this end, for any
$t\in[0,T)$, we take $(\mathbf{u}(x,t), F(x,t))$ as a new initial
data of the problem \eqref{eq1.5}--\eqref{eq1.7}, and split
$\mathbf{u}(x,t)=\mathbf{u}_{1}(x,t)+\mathbf{u}_{2}(x,t)$ such that
\begin{equation*}
  \widehat{\mathbf{u}}(\xi,t)=\widehat{\mathbf{u}}1_{|\xi|>2^{N}}(\xi,t)+\widehat{\mathbf{u}}1_{|\xi|\leq2^{N}}(\xi,t)
  :=\widehat{\mathbf{u}_{1}}(\xi,t)+\widehat{\mathbf{u}_{2}}(\xi,t).
\end{equation*}
Similarly, we split $F(x,t):=F_{1}(x,t)+F_{2}(x,t)$. Since $2\leq
p<2n$, by using the properties of the Besov spaces, there exists a
sufficiently large constant $N\in\mathbb{N}$ such that
\begin{equation}\label{eq4.1}
  C_{1}\|(\mathbf{u}_{1}(t), F_{1}(t))\|_{\dot{B}^{-1+n/p}_{p,r}}\leq
  \frac{\varepsilon}{2}.
\end{equation}
On the other hand,  if we choose $\tilde{T}>t$ such that
\begin{equation*}
  \tilde{T}-t\leq \Big(\frac{\varepsilon}{C_{2}2^{1+(2N)/q}M}\Big)^{q}:=T_{\varepsilon},
\end{equation*}
then we can obtain $\|(e^{t\Delta}\mathbf{u}_{2},
e^{t\Delta}F_{2})\|_{\mathcal{X}_{t+T_{\varepsilon}}}\leq
\frac{\varepsilon}{2}$. This result together with  \eqref{eq4.1}, by
Proposition \ref{pro3.2}, yield that there exists a constant
$T_{\varepsilon}$ depending only on $\varepsilon$ and $M$ such that
for any $t\in[0,T)$, the problem \eqref{eq1.5}--\eqref{eq1.7} has a
solution on the time interval $[t,t+T_{\varepsilon})$. By the
uniqueness we know that all solutions obtained in this way are equal
in their common existence interval, so that the solution can be
extended to the time interval $[0,T+T_{\varepsilon})$. That is to
say $T^{*}>T$, we complete the proof of Theorem \ref{th1.3}.

\section{Stability and weak-strong uniqueness}

The aim of this section is to prove Theorem \ref{th1.4}. Let us
recall the definition of weak solutions to the system
\eqref{eq1.5}--\eqref{eq1.7}.
\begin{definition}\label{def5.1}
The vector-valued function $(\mathbf{u}, F)$ is called a weak
solution of \eqref{eq1.5}--\eqref{eq1.7} on
$\mathbb{R}^{n}\times(0,T)$ if it satisfies the following
conditions:
\begin{itemize}
\item [(1)] $(\mathbf{u},F)\in L^{\infty}(0,T;
L^{2}(\mathbb{R}^{n}))\cap
L^{2}(0,T;\dot{H}^{1}(\mathbb{R}^{n})):=(\mathcal{WS})$, where
$\dot{H}^{1}(\mathbb{R}^{n})=\dot{B}^{1}_{2,2}(\mathbb{R}^{n})$ is
the usual homogeneous Sobolev space.

\item [(2)]$(\mathbf{u},F)$ satisfies the system
\eqref{eq1.5}--\eqref{eq1.7} in the distributional sense, i.e.,
$\diva\mathbf{u}=0$ in the distributional sense and for all
$\mathbf{v}\in C_{0}^{\infty}(\mathbb{R}^{n}\times(0,T))$ and $G\in
C_{0}^{\infty}(\mathbb{R}^{n}\times(0,T))$ with $\diva\mathbf{v}=0$,
we have
\begin{align*}
  \int^{T}_{0}\int_{\mathbb{R}^{n}} \mathbf{u}\partial_{t}\mathbf{v} dxdt
  &-\int^{T}_{0}\int_{\mathbb{R}^{n}}\nabla\mathbf{u}:\nabla\mathbf{v}dxdt
  -\int^{T}_{0}\int_{\mathbb{R}^{n}}\mathbf{u}\cdot\nabla\mathbf{u}\cdot\mathbf{v}dxdt\\
  &=-\int^{T}_{0}\int_{\mathbb{R}^{n}}F^{T}F:\nabla\mathbf{v}dxdt,
\end{align*}
and
\begin{align*}
  \int^{T}_{0}\int_{\mathbb{R}^{n}} F:\partial_{t}G dxdt
  &-\int^{T}_{0}\int_{\mathbb{R}^{n}}\nabla F:\nabla Gdxdt
  -\int^{T}_{0}\int_{\mathbb{R}^{n}}\mathbf{u}\cdot\nabla F: Gdxdt\\
  &=\int^{T}_{0}\int_{\mathbb{R}^{n}}F\nabla\mathbf{u}:Gdxdt.
\end{align*}

\item [(3)] The following energy inequality holds:
\begin{align*}
  \int_{\mathbb{R}^{n}}(|\mathbf{u}(t)|^{2}+|F(t)|^{2})dx+2\int_{0}^{t}\int_{\mathbb{R}^{n}}(|\nabla\mathbf{u}|^{2}+|\nabla
  F|^{2})dxd\tau\leq
  \int_{\mathbb{R}^{n}}(|\mathbf{u}_{0}|^{2}+|F_{0}|^{2})dx.
\end{align*}
\end{itemize}
\end{definition}

\noindent\textit{Remark 5.1.} Formally, taking
$\mathbf{v}=\mathbf{u}$ and $G=F$, and adding them together, we get
\begin{align*}
  \frac{1}{2}\frac{d}{dt}\int_{\mathbb{R}^{n}}(|\mathbf{u}(t)|^{2}+|F(t)|^{2})dx+\int_{\mathbb{R}^{n}}(|\nabla\mathbf{u}|^{2}+|\nabla
  F|^{2})dx=0,
\end{align*}
which implies the above energy inequality. Here we have used the
fact $AB:C=A:CB^{T}=B:A^{T}C$ for any three $n\times n$ matrixes
$A$, $B$ and $C$,

Let $(\tilde{\mathbf{u}}_{0},\tilde{F}_{0})\in
L^{2}(\mathbb{R}^{n})$, $(\mathbf{u}_{0},F_{0})\in
L^{2}(\mathbb{R}^{n})$, and we denote by $(\tilde{\mathbf{u}},
\tilde{F})$ and $(\mathbf{u}, F)$ two weak solutions in the space
$\mathcal{WS}$ associated with the initial data
$(\tilde{\mathbf{u}}_{0},\tilde{F}_{0})$ and
$(\mathbf{u}_{0},F_{0})$, respectively.  Assume that $(\mathbf{u},
F)\in L^{q}(0,T;\dot{B}^{-1+n/p+2/q}_{p,q}(\mathbb{R}^{n}))$, where
$2\leq p<\infty$ and $2<q<\infty$ satisfying
$\frac{n}{p}+\frac{2}{q}>1$. Obviously, the above energy inequality
yields that
\begin{align}\label{eq5.1}
  &\|(\tilde{\mathbf{u}}(t), \tilde{F}(t))\|_{L^{2}}^{2}+2\int_{0}^{t}\|(\nabla\tilde{\mathbf{u}}(\tau),
  \nabla \tilde{F}(\tau))\|_{L^{2}}^{2}d\tau\leq \|(\tilde{\mathbf{u}}_{0}, \tilde{F}_{0})\|_{L^{2}}^{2};\\
\label{eq5.2}
  &\|(\mathbf{u}(t),F(t))\|_{L^{2}}^{2}+2\int_{0}^{t}\|(\nabla\mathbf{u}(\tau),\nabla F(\tau))\|_{L^{2}}^{2}d\tau
  \leq \|(\mathbf{u}_{0}, F_{0})\|_{L^{2}}^{2}.
\end{align}
Let $\mathbf{w}=\mathbf{u}-\tilde{\mathbf{u}}$, $E=F-\tilde{F}$,
$\mathbf{w}_{0}=\mathbf{u}_{0}-\tilde{\mathbf{u}}_{0}$ and
$E_{0}=F_{0}-\tilde{F}_{0}$. To prove Theorem \ref{th1.4}, it
suffices to prove the following result:
\begin{proposition}\label{pro5.2}
Under the hypotheses of Theorem \ref{th1.4}, we have
\begin{align}\label{eq5.3}
  \|(\mathbf{w}(t), E(t))\|_{L^{2}}^{2}&+2\int^{t}_{0}\|(\nabla\mathbf{w}(\tau), \nabla E(\tau))\|_{L^{2}}^{2}d\tau
  \nonumber\\&\leq\|(\mathbf{w}_{0}, E_{0})\|_{L^{2}}^{2}
  \times\exp\Big(C\int^{t}_{0}\|(\mathbf{u}(\tau),
  F(\tau))\|_{\dot{B}^{-1+n/p+2/q}_{p,q}}^{q}d\tau\Big).
\end{align}
\end{proposition}

Note that by \eqref{eq5.1} and \eqref{eq5.2},
\begin{align*}
  \|(\mathbf{w}(t)&, E(t))\|_{L^{2}}^{2}+2\int^{t}_{0}\|(\nabla\mathbf{w}(\tau), \nabla E(\tau))\|_{L^{2}}^{2}d\tau
  \!=\!\|(\mathbf{u}(t), F(t))\|_{L^{2}}^{2}+\|(\tilde{\mathbf{u}}(t),
  \tilde{F}(t))\|_{L^{2}}^{2}\\&+2\int^{t}_{0}\|(\nabla\mathbf{u}(\tau), \nabla
  F(\tau))\|_{L^{2}}^{2}d\tau+2\int^{t}_{0}\|(\nabla\tilde{\mathbf{u}}(\tau), \nabla
  \tilde{F}(\tau))\|_{L^{2}}^{2}d\tau-2\big(\mathbf{u}(t)|\tilde{\mathbf{u}}(t)\big)\\&-2\big(F(t)|
  \tilde{F}(t)\big)-4\int_{0}^{t}\big(\nabla\mathbf{u}(\tau)|
  \nabla\tilde{\mathbf{u}}(\tau)\big)d\tau-4\int_{0}^{t}\big(\nabla
  F(\tau)|
  \nabla\tilde{F}(\tau)\big)d\tau\\
  &\leq\|(\mathbf{u}_{0}, F_{0})\|_{L^{2}}^{2}+\|(\tilde{\mathbf{u}}_{0}, \tilde{F}_{0})\|_{L^{2}}^{2}
  -2\big(\mathbf{u}(t)| \tilde{\mathbf{u}}(t)\big)-2\big(F(t)|
  \tilde{F}(t)\big)\\&-4\int_{0}^{t}\big(\nabla\mathbf{u}(\tau)|
  \nabla\tilde{\mathbf{u}}(\tau)\big)d\tau-4\int_{0}^{t}\big(\nabla
  F(\tau)|
  \nabla\tilde{F}(\tau)\big)d\tau.
\end{align*}
Here $(\cdot|\cdot)$ denotes the scalar product in
$L^{2}(\mathbb{R}^{2})$. In order to prove Proposition \ref{pro5.2},
we need to introduce the following lemma.

\begin{lemma}\label{le5.3}
Under the hypothese of Theorem \ref{th1.4}, the following equality
holds for all $t\leq T$,
\begin{align}\label{eq5.4}
  \big(\mathbf{u}(t)&| \tilde{\mathbf{u}}(t)\big)+\big(F(t)|
  \tilde{F}(t)\big)+2\int_{0}^{t}\big(\nabla\mathbf{u}(\tau)|
  \nabla\tilde{\mathbf{u}}(\tau)\big)d\tau+2\int_{0}^{t}\big(\nabla
  F(\tau)|
  \nabla\tilde{F}(\tau)\big)d\tau\nonumber\\
  &=(\mathbf{u}_{0}|\tilde{\mathbf{u}}_{0})+(F_{0}|\tilde{F}_{0})
  -\int_{0}^{t}\int_{\mathbb{R}^{n}}\mathbf{w}\cdot\nabla\mathbf{w}\cdot\mathbf{u}dxd\tau
  +\int_{0}^{t}\int_{\mathbb{R}^{n}}\tilde{F}^{T}\tilde{F}:\nabla\mathbf{u}dxd\tau\nonumber\\
  &+\int_{0}^{t}\int_{\mathbb{R}^{n}}F^{T}F:\nabla\mathbf{u}dxd\tau
  -\int_{0}^{t}\int_{\mathbb{R}^{n}}\nabla\mathbf{u}:(\tilde{F}^{T}F+F^{T}\tilde{F})dxd\tau\nonumber\\
  &-\int_{0}^{t}\int_{\mathbb{R}^{n}}\mathbf{w}\cdot\nabla F:\tilde{F}dxd\tau
  +\int_{0}^{t}\int_{\mathbb{R}^{n}}F:\tilde{F}\nabla\mathbf{w}dxd\tau.
\end{align}
\end{lemma}

We shall use Lemma 1.1 in  \cite{GP02} to prove Lemma \ref{le5.3}.

\begin{lemma}\label{le5.4} {\rm {(\cite{GP02})}}
Let $n\geq 2$, $2\leq p<\infty$ and $2<q<\infty$ such that
$\frac{n}{p}+\frac{2}{q}>1$. Then for every $T>0$, the trilinear
form
\begin{equation*}
  (\mathbf{u},\mathbf{v},\mathbf{w})\in\mathcal{WS}\times\mathcal{WS}\times
  L^{q}(0,T;\dot{B}^{-1+n/p+2/q}_{p,q}(\mathbb{R}^{n}))\mapsto\int_{0}^{T}\int_{\mathbb{R}^{n}}\mathbf{u}\cdot\nabla
  \mathbf{v}\cdot
  \mathbf{w}dxdt
\end{equation*}
is continuous. In particular, the following estimate holds:
\begin{align}\label{eq5.5}
  &\Big|\int_{0}^{T}\int_{\mathbb{R}^{n}}\mathbf{u}\cdot\nabla\mathbf{
  v}\cdot\mathbf{w}dxdt\Big|\nonumber\\&\leq
  C\|\mathbf{u}\|_{L^{\infty}(0,T;L^{2})}^{2/q}\|\nabla
  \mathbf{u}\|_{L^{2}(0,T;L^{2})}^{1-2/q}\|\nabla
  \mathbf{v}\|_{L^{2}(0,T;L^{2})}\|\mathbf{w}\|_{L^{q}(0,T;\dot{B}^{-1+n/p+2/q}_{p,q})}\nonumber\\
  &+\|\nabla
  \mathbf{u}\|_{L^{2}(0,T;L^{2})}\|\mathbf{v}\|_{L^{\infty}(0,T;L^{2})}^{2/q}\|\nabla
  \mathbf{v}\|_{L^{2}(0,T;L^{2})}^{1-2/q}\|\mathbf{w}\|_{L^{q}(0,T;\dot{B}^{-1+n/p+2/q}_{p,q})}\nonumber\\
  &+\|\mathbf{u}\|_{L^{\infty}(0,T;L^{2})}^{1/q}\|\nabla
  \mathbf{u}\|_{L^{2}(0,T;L^{2})}^{1-1/q}\|\mathbf{v}\|_{L^{\infty}(0,T;L^{2})}^{1/q}\|\nabla
  \mathbf{v}\|_{L^{2}(0,T;L^{2})}^{1-1/q}\|\mathbf{w}\|_{L^{q}(0,T;\dot{B}^{-1+n/p+2/q}_{p,q})}.
\end{align}
\end{lemma}
\noindent\textit{Remark 5.2.} Checking the proof of Lemma 1.1 in
\cite{GP02} in detail, we find that the special structure of the
Navier-Stokes equations ($\text{div } \mathbf{u}=0$) was not used,
and \eqref{eq5.5} holds in both scalar and vector cases.

\textbf{The proof of Lemma \ref{le5.3}.}

Let us consider two smooth sequences of
$(\{\tilde{\mathbf{u}}_{n}\}, \{\tilde{F}_{n}\})$ ($\text{div
}\tilde{\mathbf{u}}_{n}=0$) and $(\{\mathbf{u}_{n}\}, \{F_{n}\})$
($\text{div }\mathbf{u}_{n}=0$) such that
\begin{equation*}
\begin{cases}
  \lim_{n\rightarrow\infty}(\tilde{\mathbf{u}}_{n},
  \tilde{F}_{n})=(\tilde{\mathbf{u}},\tilde{F})\ \ \text{in}\ \
  L^{2}(0,T; \dot{H}^{1}(\mathbb{R}^{n})),\\
  \lim_{n\rightarrow\infty}(\tilde{\mathbf{u}}_{n},
  \tilde{F}_{n})=(\tilde{\mathbf{u}},\tilde{F})\ \ \text{weakly-star in}\ \
  L^{\infty}(0,T; L^{2}(\mathbb{R}^{n})),
\end{cases}
\end{equation*}
and
\begin{equation*}
\begin{cases}
  \lim_{n\rightarrow\infty}(\mathbf{u}_{n},
  F_{n})=(\mathbf{u},F)\ \ \text{in}\ \
  L^{2}(0,T; \dot{H}^{1}(\mathbb{R}^{n}))\cap
  L^{q}(0,T;\dot{B}^{-1+n/p+2/q}_{p,q}(\mathbb{R}^{n})),\\
  \lim_{n\rightarrow\infty}(\mathbf{u}_{n},
  F_{n})=(\mathbf{u},F)\ \ \text{weakly-star in}\ \
  L^{\infty}(0,T; L^{2}(\mathbb{R}^{n})).
\end{cases}
\end{equation*}
We split the proof into the following two steps.

\textbf{Step 1.} Taking the scalar product with
$\tilde{\mathbf{u}}_{n}$ and $\mathbf{u}_{n}$ of the equation
\eqref{eq1.5} on $\mathbf{u}$ and $\tilde{\mathbf{u}}$ respectively,
after integration in time and integration by parts in the space
variables, we get
\begin{equation}\label{eq5.6}
  \int_{0}^{t}\Big((\partial_{\tau}\mathbf{u}| \tilde{\mathbf{u}}_{n})+(\nabla\mathbf{u}| \nabla\tilde{\mathbf{u}}_{n})
  +(\mathbf{u}\cdot\nabla\mathbf{u}|\tilde{\mathbf{u}}_{n})+(\diva(F^{T}F)|
  \tilde{\mathbf{u}}_{n})\Big)d\tau=0
\end{equation}
and
\begin{equation}\label{eq5.7}
  \int_{0}^{t}\Big((\partial_{\tau}\tilde{\mathbf{u}}| \mathbf{u}_{n})+(\nabla\tilde{\mathbf{u}}| \nabla\mathbf{u}_{n})
  +(\tilde{\mathbf{u}}\cdot\nabla\tilde{\mathbf{u}}|\mathbf{u}_{n})+(\diva(\tilde{F}^{T}\tilde{F})|
  \mathbf{u}_{n})\Big)d\tau=0.
\end{equation}
Note that we have assumed that both $\nabla\mathbf{u}_{n}$ and
$\nabla\tilde{\mathbf{u}}_{n}$ converge in
$L^{2}(0,T;L^{2}(\mathbb{R}^{n}))$ towards $\nabla\mathbf{u}$ and
$\nabla\tilde{\mathbf{u}}$ respectively, it is obvious that
\begin{equation}\label{eq5.8}
  \lim_{n\rightarrow\infty}\Big(\int_{0}^{t}(\nabla\mathbf{u}| \nabla\tilde{\mathbf{u}}_{n})d\tau+\int_{0}^{t}(\nabla\tilde{\mathbf{u}}| \nabla\mathbf{u}_{n})
  d\tau\Big)=2\int_{0}^{t}(\nabla\mathbf{u}|
  \nabla\tilde{\mathbf{u}})d\tau.
\end{equation}
Since $\mathbf{u}_{n}$ converges to $\mathbf{u}$ in
$L^{q}(0,T;\dot{B}^{-1+n/p+2/q}_{p,q}(\mathbb{R}^{n}))$, by Lemma
\ref{le5.4}, one obtains that
\begin{equation}\label{eq5.9}
  \lim_{n\rightarrow\infty}\int_{0}^{t}(\tilde{\mathbf{u}}\cdot\nabla\tilde{\mathbf{u}}|\mathbf{u}_{n})d\tau
  =\int_{0}^{t}(\tilde{\mathbf{u}}\cdot\nabla\tilde{\mathbf{u}}|\mathbf{u})d\tau.
\end{equation}
Due to the fact $\text{div }\mathbf{u}=0$, this yields, by Lemma
\ref{le5.4} again,
\begin{equation}\label{eq5.10}
  \lim_{n\rightarrow\infty}\int_{0}^{t}(\mathbf{u}\cdot\nabla\mathbf{u}|\tilde{\mathbf{u}}_{n})d\tau
  =-\lim_{n\rightarrow\infty}\int_{0}^{t}(\mathbf{u}\cdot\nabla\tilde{\mathbf{u}}_{n}|\mathbf{u})d\tau
  =-\int_{0}^{t}(\mathbf{u}\cdot\nabla\tilde{\mathbf{u}}|\mathbf{u})d\tau.
\end{equation}
Applying \eqref{eq5.5}, it is also obvious that the following two
results hold:
\begin{align}\label{eq5.11}
  \lim_{n\rightarrow\infty}&\int_{0}^{t}(\diva(F^{T}F)|
  \tilde{\mathbf{u}}_{n})d\tau=-\lim_{n\rightarrow\infty}\int_{0}^{t}(F^{T}F|
  \nabla\tilde{\mathbf{u}}_{n})d\tau\nonumber\\
  &=-\int_{0}^{t}(F^{T}F| \nabla\tilde{\mathbf{u}})d\tau=\int_{0}^{t}(\diva(F^{T}F)|\tilde{\mathbf{u}})d\tau.
\end{align}
and
\begin{align}\label{eq5.12}
  \lim_{n\rightarrow\infty}&\int_{0}^{t}(\diva(\tilde{F}^{T}\tilde{F})\big|
  \mathbf{u}_{n})d\tau
  =\lim_{n\rightarrow\infty}\int_{0}^{t}\Big(\sum_{i=1}^{n}(\partial_{x_{i}}\tilde{F}^{T}\tilde{F}+\tilde{F}^{T}\partial_{x_{i}}\tilde{F})|
  \mathbf{u}_{n}\Big)d\tau\nonumber\\
  &=\int_{0}^{t}\Big(\sum_{i=1}^{n}(\partial_{x_{i}}\tilde{F}^{T}\tilde{F}+\tilde{F}^{T}\partial_{x_{i}}\tilde{F})\big|
  \mathbf{u}\Big)d\tau=\int_{0}^{t}(\diva(\tilde{F}^{T}\tilde{F})|
  \mathbf{u})d\tau.
\end{align}
Since $\partial_{t}\tilde{\mathbf{u}}=\Delta\tilde{\mathbf{u}}
-\mathbb{P}\tilde{\mathbf{u}}\cdot\nabla\tilde{\mathbf{u}}-\mathbb{P}\diva(\tilde{F}^{T}\tilde{F})$
holds in the sense of distribution, the estimates
\eqref{eq5.8}--\eqref{eq5.12} imply in particular that
\begin{align*}
  \lim_{n\rightarrow\infty}\int_{0}^{t}(\partial_{\tau}\tilde{\mathbf{u}}|
  \mathbf{u}_{n})d\tau&=-\lim_{n\rightarrow\infty}\int_{0}^{t}\Big((\nabla\tilde{\mathbf{u}}| \nabla\mathbf{u}_{n})
  +(\tilde{\mathbf{u}}\cdot\nabla\tilde{\mathbf{u}}|\mathbf{u}_{n})+(\diva(\tilde{F}^{T}\tilde{F})|
  \mathbf{u}_{n})\Big)d\tau\nonumber\\&=-\int_{0}^{t}\Big((\nabla\tilde{\mathbf{u}}|\nabla\mathbf{u})
  +(\tilde{\mathbf{u}}\cdot\nabla\tilde{\mathbf{u}}|\mathbf{u})+(\diva(\tilde{F}^{T}\tilde{F})|
  \mathbf{u})\Big)d\tau\nonumber\\&=\int_{0}^{t}(\partial_{\tau}\tilde{\mathbf{u}}|
  \mathbf{u})d\tau.
\end{align*}
Similarly, we have
\begin{align*}
  \lim_{n\rightarrow\infty}\int_{0}^{t}(\partial_{\tau}\mathbf{u}| \tilde{\mathbf{u}}_{n})d\tau
  =\int_{0}^{t}(\partial_{\tau}\mathbf{u}| \tilde{\mathbf{u}})d\tau.
\end{align*}
Putting these estimates together, and noticing that
\begin{align}\label{eq5.13}
 \int_{0}^{t}(\partial_{\tau}\tilde{\mathbf{u}}|
  \mathbf{u})+(\partial_{\tau}\mathbf{u}| \tilde{\mathbf{u}})d\tau=(\mathbf{u}(t)|
  \tilde{\mathbf{u}}(t))-(\mathbf{u}_{0}|
  \tilde{\mathbf{u}}_{0})
\end{align}
and
\begin{align}\label{eq5.14}
 \int_{0}^{t}\Big((\tilde{\mathbf{u}}\cdot\nabla\tilde{\mathbf{u}}|\mathbf{u})-(\mathbf{u}\cdot\nabla\tilde{\mathbf{u}}|\mathbf{u})\Big)d\tau
 =\int_{0}^{t}(\mathbf{w}\cdot\nabla\mathbf{w}| \mathbf{u})d\tau,
\end{align}
we find
\begin{align}\label{eq5.15}
  (\mathbf{u}(t)|
  \tilde{\mathbf{u}}(t))+&2\int_{0}^{t}(\nabla\mathbf{u}|
  \nabla\tilde{\mathbf{u}})d\tau=(\mathbf{u}_{0}|
  \tilde{\mathbf{u}}_{0})-\int_{0}^{t}(\mathbf{w}\cdot\nabla\mathbf{w}| \mathbf{u})d\tau\nonumber\\
  &+\int_{0}^{t}\int_{\mathbb{R}^{n}}\tilde{F}^{T}\tilde{F}:\nabla\mathbf{u}dxd\tau
  +\int_{0}^{t}\int_{\mathbb{R}^{n}}F^{T}F:\nabla\mathbf{u}dxd\tau.
\end{align}

\textbf{Step 2.} In this step we derive the estimate for $F$ and
$\tilde{F}$. Proceeding the same way as \eqref{eq5.6} and
\eqref{eq5.7}, we obtain that
\begin{equation}\label{eq5.16}
  \int_{0}^{t}\Big((\partial_{\tau}F| \tilde{F}_{n})+(\nabla F|\nabla\tilde{F}_{n})
  +(\mathbf{u}\cdot\nabla F|\tilde{F}_{n})+(F\nabla\mathbf{u}|
  \tilde{F}_{n})\Big)d\tau=0
\end{equation}
and
\begin{equation}\label{eq5.17}
  \int_{0}^{t}\Big((\partial_{\tau}\tilde{F}| F_{n})+(\nabla\tilde{F}| \nabla F_{n})
  +(\tilde{\mathbf{u}}\cdot\nabla\tilde{F}|F_{n})+(\tilde{F}\nabla\tilde{\mathbf{u}}|
  F_{n})\Big)d\tau=0.
\end{equation}
Since we have assumed that both $\nabla F_{n}$ and
$\nabla\tilde{F}_{n}$ respectively converge to $\nabla F$ and
$\nabla\tilde{F}$ in $L^{2}(0,T;L^{2}(\mathbb{R}^{n}))$ and $F_{n}$
converge to $F$ in
$L^{q}(0,T;\dot{B}^{-1+n/p+2/q}_{p,q}(\mathbb{R}^{n}))$. Then by
using  \eqref{eq5.5} and the fact $\text{div }\mathbf{u}=0$, it is
clear that
\begin{equation}\label{eq5.18}
  \lim_{n\rightarrow\infty}\Big(\int_{0}^{t}(\nabla F| \nabla\tilde{F}_{n})d\tau+\int_{0}^{t}(\nabla\tilde{F}| \nabla F_{n})
  d\tau\Big)=2\int_{0}^{t}(\nabla F|
  \nabla\tilde{F})d\tau,
\end{equation}
\begin{equation}\label{eq5.19}
  \lim_{n\rightarrow\infty}\int_{0}^{t}(\tilde{\mathbf{u}}\cdot\nabla\tilde{F}|F_{n})d\tau
  =\int_{0}^{t}(\tilde{\mathbf{u}}\cdot\nabla\tilde{F}|F)d\tau,
\end{equation}
\begin{equation}\label{eq5.20}
  \lim_{n\rightarrow\infty}\int_{0}^{t}(\mathbf{u}\cdot\nabla F|\tilde{F}_{n})d\tau
  =-\lim_{n\rightarrow\infty}\int_{0}^{t}(\mathbf{u}\cdot\nabla\tilde{F}_{n}|F)d\tau
  =-\int_{0}^{t}(\mathbf{u}\cdot\nabla\tilde{F}|F)d\tau,
\end{equation}
\begin{equation}\label{eq5.21}
  \lim_{n\rightarrow\infty}\int_{0}^{t}(\tilde{F}\nabla\tilde{\mathbf{u}}|
  F_{n})d\tau
  =\int_{0}^{t}(\tilde{F}\nabla\tilde{\mathbf{u}}|
  F)d\tau,
\end{equation}
and
\begin{equation}\label{eq5.22}
  \lim_{n\rightarrow\infty}\int_{0}^{t}(F\nabla\mathbf{u}|
  \tilde{F}_{n})d\tau
  =\int_{0}^{t}(F\nabla\mathbf{u}|
  \tilde{F})d\tau.
\end{equation}
The last identity holds because $\nabla\tilde{F}_{n}$ converge to
$\nabla\tilde{F}$ in $L^{2}(0,T; L^{2}(\mathbb{R}^{n}))$ and
$\{\tilde{F}_{n}\}$ is bounded in $L^{\infty}(0,T;
L^{2}(\mathbb{R}^{n}))$, which was ensured by the Banach-Steinhaus
theorem due to $\tilde{F}_{n}$ weakly-star converge to $\tilde{F}$
in $L^{\infty}(0,T; L^{2}(\mathbb{R}^{n}))$. Hence, combining the
estimates \eqref{eq5.18}--\eqref{eq5.22}, as derivation of estimate
\eqref{eq5.13} and \eqref{eq5.14}, we have
\begin{align}\label{eq5.23}
  \lim_{n\rightarrow\infty}\int_{0}^{t}((\partial_{\tau}F| \tilde{F}_{n})d\tau=\int_{0}^{t}(\partial_{\tau}F| \tilde{F})d\tau
\end{align}
and
\begin{align}\label{eq5.24}
  \lim_{n\rightarrow\infty}\int_{0}^{t}(\partial_{\tau}\tilde{F}| F_{n})d\tau
  =\int_{0}^{t}(\partial_{\tau}\tilde{F}| F)d\tau.
\end{align}
It is obvious that
\begin{align*}
  \int_{0}^{t}(\partial_{\tau}F| \tilde{F})+(\partial_{\tau}\tilde{F}| F)d\tau=(F(t)|
  \tilde{F}(t))-(F_{0}|
  \tilde{F}_{0}).
\end{align*}
Moreover, since $\text{div }\mathbf{u}=0$, we have
\begin{align*}
  \int_{0}^{t}\int_{\mathbb{R}^{n}}\Big(\tilde{\mathbf{u}}\cdot\nabla\tilde{F}:F+\tilde{\mathbf{u}}\cdot\nabla
  F:\tilde{F}\Big)dxd\tau=\int_{0}^{t}\int_{\mathbb{R}^{n}}\tilde{\mathbf{u}}\cdot\nabla(\tilde{F}:F)dxd\tau=0
\end{align*}
and
\begin{align*}
 \tilde{F}\nabla\mathbf{u}:F+\tilde{F}:F\nabla\mathbf{u}=\nabla\mathbf{u}:(\tilde{F}^{T}F+F^{T}\tilde{F}).
\end{align*}
These two facts imply that
\begin{align*}
  \int_{0}^{t}\int_{\mathbb{R}^{n}}&\Big(\mathbf{u}\cdot\nabla F:\tilde{F}
  +\tilde{\mathbf{u}}\cdot\nabla\tilde{F}:F
  +F\nabla\mathbf{u}:\tilde{F}
  +\tilde{F}\nabla\tilde{\mathbf{u}}:F\Big)dxd\tau\\
  =&\int_{0}^{t}\int_{\mathbb{R}^{n}}\Big(\nabla\mathbf{u}:(\tilde{F}^{T}F+F^{T}\tilde{F})+(\mathbf{u}-\tilde{\mathbf{u}})\cdot\nabla
  F:\tilde{F}-F:\tilde{F}\nabla(\mathbf{u}-\tilde{\mathbf{u}})
  \Big)dxd\tau.
\end{align*}
Finally, putting all above estimates together yield
\begin{align}\label{eq5.25}
  (F(t)&| \tilde{F}(t))+2\int_{0}^{t}(\nabla F|
  \nabla\tilde{F})d\tau=(F_{0}|
  \tilde{F}_{0})\nonumber\\&-\int_{0}^{t}\int_{\mathbb{R}^{n}}\Big(\nabla\mathbf{u}:(\tilde{F}^{T}F+F^{T}\tilde{F})
  +\mathbf{w}\cdot\nabla
  F:\tilde{F}-F:\tilde{F}\nabla\mathbf{w}
  \Big)dxd\tau.
\end{align}
It is clear that \eqref{eq5.4} follows from \eqref{eq5.15} and
\eqref{eq5.25}. This completes the proof of Lemma \ref{le5.3}.

\textbf{The proof of Proposition \ref{pro5.2}.}

Note that $\int_{0}^{t}\int_{\mathbb{R}^{n}}\mathbf{w}\cdot\nabla
F:Fdxd\tau=0$ since $\text{div }\mathbf{w}=0$, then by Lemma
\ref{le5.3}, we get
\begin{align}\label{eq5.26}
  \|(\mathbf{w}(t)&, E(t))\|_{L^{2}}^{2}+2\int^{t}_{0}\|(\nabla\mathbf{w}(\tau), \nabla
  E(\tau))\|_{L^{2}}^{2}d\tau
  \leq\|(\mathbf{u}_{0}, F_{0})\|_{L^{2}}^{2}+\|(\tilde{\mathbf{u}}_{0},
  \tilde{F}_{0})\|_{L^{2}}^{2}\nonumber\\&-2(\mathbf{u}_{0}|\tilde{\mathbf{u}}_{0})-2(F_{0}|\tilde{F}_{0})
  +2\int_{0}^{t}\int_{\mathbb{R}^{n}}\mathbf{w}\cdot\nabla\mathbf{w}\cdot\mathbf{u}dxd\tau
  -2\int_{0}^{t}\int_{\mathbb{R}^{n}}\tilde{F}^{T}\tilde{F}:\nabla\mathbf{u}dxd\tau\nonumber\\
  &-2\int_{0}^{t}\int_{\mathbb{R}^{n}}F^{T}F:\nabla\tilde{\mathbf{u}}dxd\tau
  +2\int_{0}^{t}\int_{\mathbb{R}^{n}}\nabla\mathbf{u}:(\tilde{F}^{T}F+F^{T}\tilde{F})dxd\tau\nonumber\\
  &+2\int_{0}^{t}\int_{\mathbb{R}^{n}}\mathbf{w}\cdot\nabla F:\tilde{F}dxd\tau
  -2\int_{0}^{t}\int_{\mathbb{R}^{n}}F:\tilde{F}\nabla\mathbf{w}dxd\tau\nonumber\\
  &\leq\|(\mathbf{w}_{0}, E_{0})\|_{L^{2}}^{2}+2\int_{0}^{t}\int_{\mathbb{R}^{n}}\mathbf{w}\cdot\nabla\mathbf{w}\cdot\mathbf{u}dxd\tau
  -2\int_{0}^{t}\int_{\mathbb{R}^{n}}E^{T}E:\nabla\mathbf{u}dxd\tau\nonumber\\
  &-2\int_{0}^{t}\int_{\mathbb{R}^{n}}\mathbf{w}\cdot\nabla F:Edxd\tau
  +2\int_{0}^{t}\int_{\mathbb{R}^{n}}E^{T}F:\nabla\mathbf{w}dxd\tau.
\end{align}
Using the similar argument as the proof of Lemma \ref{le5.4} (see
\cite{GP02}), we obtain that
\begin{align}\label{eq5.27}
  \Big|\int_{0}^{t}\int_{\mathbb{R}^{n}}\mathbf{w}\cdot\nabla\mathbf{w}\cdot&\mathbf{u}dxd\tau\Big|
  \leq C\int_{0}^{t}\|\mathbf{w}\|_{L^{2}}^{2/q}\|\nabla
  \mathbf{w}\|_{L^{2}}^{2-2/q}\|\mathbf{u}\|_{\dot{B}^{-1+n/p+2/q}_{p,q}}d\tau\nonumber\\
  &\leq\frac{1}{2}\int^{t}_{0}\|\nabla\mathbf{w}\|_{L^{2}}^{2}d\tau
  +C\int_{0}^{t}\|\mathbf{w}\|_{L^{2}}^{2}
  \|\mathbf{u}\|_{\dot{B}^{-1+n/p+2/q}_{p,q}}^{q}d\tau;
\end{align}
\begin{align}\label{eq5.28}
  \Big|\int_{0}^{t}\int_{\mathbb{R}^{n}}E^{T}E:&\nabla\mathbf{u}dxd\tau\Big|
  =\Big|-\int_{0}^{t}\int_{\mathbb{R}^{n}}\diva(E^{T}E)\cdot\mathbf{u}dxd\tau\Big|\nonumber\\
  &\leq C\int_{0}^{t}\|E\|_{L^{2}}^{2/q}\|\nabla
  E\|_{L^{2}}^{2-2/q}\|\mathbf{u}\|_{\dot{B}^{-1+n/p+2/q}_{p,q}}d\tau\nonumber\\
  &\leq\frac{1}{2}\int^{t}_{0}\|\nabla E\|_{L^{2}}^{2}d\tau
  +C\int_{0}^{t}\|E\|_{L^{2}}^{2}\|\mathbf{u}\|_{\dot{B}^{-1+n/p+2/q}_{p,q}}^{q}d\tau;
\end{align}
\begin{align}\label{eq5.29}
  \Big|\int_{0}^{t}\int_{\mathbb{R}^{n}}\mathbf{w}\cdot\nabla &F:Edxd\tau\Big|
  =\Big|\int_{0}^{t}\int_{\mathbb{R}^{n}}(\mathbf{w}\otimes F)\cdot\nabla
  Edxd\tau\Big|\nonumber\\
  &\leq\frac{1}{2}\int^{t}_{0}\|(\nabla\mathbf{w},\nabla E)\|_{L^{2}}^{2}d\tau
  +C\int_{0}^{t}\|(\mathbf{w}, E)\|_{L^{2}}^{2}\|F\|_{\dot{B}^{-1+n/p+2/q}_{p,q}}^{q}d\tau;
\end{align}
\begin{align}\label{eq5.30}
  \Big|\int_{0}^{t}\int_{\mathbb{R}^{n}}E^{T}F:\nabla\mathbf{w}dxd\tau\Big|
  &\leq\frac{1}{2}\int^{t}_{0}\|(\nabla\mathbf{w},\nabla E)\|_{L^{2}}^{2}d\tau\nonumber\\
  &+C\int_{0}^{t}\|(\mathbf{w}, E)\|_{L^{2}}^{2}\|F\|_{\dot{B}^{-1+n/p+2/q}_{p,q}}^{q}d\tau.
\end{align}
where $\otimes$ denotes the tensor product. Returning back to the
estimate \eqref{eq5.26} and putting \eqref{eq5.27}--\eqref{eq5.30}
together yield that
\begin{align}\label{eq5.31}
  \|(\mathbf{w}(t),&E(t))\|_{L^{2}}^{2}+2\int^{t}_{0}\|(\nabla\mathbf{w}(\tau), \nabla
  E(\tau))\|_{L^{2}}^{2}d\tau \leq\|(\mathbf{w}_{0}, E_{0})\|_{L^{2}}^{2}\nonumber\\
  &+C\int_{0}^{t}(\|\mathbf{w}\|_{L^{2}}^{2}
  +\|E\|_{L^{2}}^{2})(\|\mathbf{u}\|_{\dot{B}^{-1+n/p+2/q}_{p,q}}^{q}+\|F\|_{\dot{B}^{-1+n/p+2/q}_{p,q}}^{q})d\tau.
\end{align}
This estimate together with Gronwall's inequality yield the desired
estimate \eqref{eq5.3} immediately.  We complete the proof of
Proposition \ref{pro5.2}. $\hfill\Box$

\end{document}